\documentclass[11pt]{article}
\usepackage[margin=1in]{geometry} 
\usepackage{amssymb,amsfonts,amsmath,bbm,mathrsfs,stmaryrd}
\usepackage{xcolor}
\usepackage{url}

\usepackage{enumerate}

\usepackage[colorlinks,
             linkcolor=black!75!red,
             citecolor=blue,
             pdftitle={},
             pdfproducer={pdfLaTeX},
             pdfpagemode=None,
             bookmarksopen=true
             bookmarksnumbered=true]{hyperref}

\usepackage{tikz}
\usetikzlibrary{arrows,calc,decorations.pathreplacing,decorations.markings,intersections,shapes.geometric,through,fit,shapes.symbols,positioning,decorations.pathmorphing}

\usepackage{braket}

\usepackage[amsmath,thmmarks,hyperref]{ntheorem}
\usepackage{cleveref}

\creflabelformat{enumi}{#2(#1)#3}

\crefname{section}{Section}{Sections}
\crefformat{section}{#2Section~#1#3} 
\Crefformat{section}{#2Section~#1#3} 

\crefname{subsection}{\S}{\S\S}
\AtBeginDocument{%
  \crefformat{subsection}{#2\S#1#3}%
  \Crefformat{subsection}{#2\S#1#3}%
}

\crefname{subsubsection}{\S}{\S\S}
\AtBeginDocument{%
  \crefformat{subsubsection}{#2\S#1#3}%
  \Crefformat{subsubsection}{#2\S#1#3}%
}

\theoremstyle{plain}

\newtheorem{lemma}{Lemma}[section]
\newtheorem{proposition}[lemma]{Proposition}

\newtheorem{theorem}[lemma]{Theorem}

\newtheorem{question}[lemma]{Question}

\theoremstyle{nonumberplain}

\theoremstyle{plain}
\theorembodyfont{\upshape}
\theoremsymbol{\ensuremath{\blacklozenge}}

\newtheorem{definition}[lemma]{Definition}

\newtheorem{remark}[lemma]{Remark}

\crefname{definition}{definition}{definitions}
\crefformat{definition}{#2definition~#1#3} 
\Crefformat{definition}{#2Definition~#1#3} 

\crefname{ex}{example}{examples}
\crefformat{example}{#2example~#1#3} 
\Crefformat{example}{#2Example~#1#3} 

\crefname{remark}{remark}{remarks}
\crefformat{remark}{#2remark~#1#3} 
\Crefformat{remark}{#2Remark~#1#3} 

\crefname{convention}{convention}{conventions}
\crefformat{convention}{#2convention~#1#3} 
\Crefformat{convention}{#2Convention~#1#3} 

\crefname{notation}{notation}{notations}
\crefformat{notation}{#2notation~#1#3} 
\Crefformat{notation}{#2Notation~#1#3} 

\crefname{table}{table}{tables}
\crefformat{table}{#2table~#1#3} 
\Crefformat{table}{#2Table~#1#3}

\crefname{lemma}{lemma}{lemmas}
\crefformat{lemma}{#2lemma~#1#3} 
\Crefformat{lemma}{#2Lemma~#1#3} 

\crefname{proposition}{proposition}{propositions}
\crefformat{proposition}{#2proposition~#1#3} 
\Crefformat{proposition}{#2Proposition~#1#3} 

\crefname{corollary}{corollary}{corollaries}
\crefformat{corollary}{#2corollary~#1#3} 
\Crefformat{corollary}{#2Corollary~#1#3} 

\crefname{theorem}{theorem}{theorems}
\crefformat{theorem}{#2theorem~#1#3} 
\Crefformat{theorem}{#2Theorem~#1#3} 

\crefname{enumi}{}{}
\crefformat{enumi}{(#2#1#3)}
\Crefformat{enumi}{(#2#1#3)}

\crefname{assumption}{assumption}{Assumptions}
\crefformat{assumption}{#2assumption~#1#3} 
\Crefformat{assumption}{#2Assumption~#1#3} 

\crefname{equation}{}{}
\crefformat{equation}{(#2#1#3)} 
\Crefformat{equation}{(#2#1#3)}

\numberwithin{equation}{section}

\theoremstyle{nonumberplain}
\theoremsymbol{\ensuremath{\blacksquare}}

\newtheorem{proof}{Proof}

\newcommand\pff[3]{\newtheorem{#1}{Proof of \Cref{#2}#3}}

\newcommand\bC{{\mathbb C}}

\newcommand\bL{{\mathbb L}}

\newcommand\bR{{\mathbb R}}
\newcommand\bS{{\mathbb S}}
\newcommand\bT{{\mathbb T}}

\newcommand\bZ{{\mathbb Z}}

\newcommand\cT{{\mathcal T}}

\DeclareMathOperator{\End}{\mathrm{End}}

\newcommand{\cat}[1]{\textsc{#1}}

\title{Shilov boundaries determine irreducible bounded symmetric domains}
\author{Alexandru Chirvasitu}

\begin{document}

\date{}

\newcommand{\Addresses}{{
  \bigskip
  \footnotesize

  \textsc{Department of Mathematics, University at Buffalo, Buffalo,
    NY 14260-2900, USA}\par\nopagebreak \textit{E-mail address}:
  \texttt{achirvas@buffalo.edu}

}}

\maketitle

\begin{abstract}
  We prove that irreducible symmetric domains are uniquely determined by the homotopy equivalence classes of their Shilov boundaries.
\end{abstract}


\vspace{.5cm}



\section*{Introduction}

Let $D\subset V\cong \bC^n$ be a circular bounded symmetric domain (see \Cref{se:prel} below for a reminder on the relevant definitions). Recall:

\begin{definition}\label{def:shil}
  The {\it Shilov (or Bergman-Shilov) boundary} $\check{S}(D)$ of a bounded symmetric domain $D\subset V$ is the smallest subset of the boundary $\partial D$ on which all functions continuous on $\overline{D}$ and holomorphic on $D$ attain their maxima.
\end{definition}

As explained in \cite[\S 3.1]{kw}, $\check{S}(D)$ is indeed well defined (i.e. always exists and is unique, for any $D$).  It has many other characterizations that will not play an important role below (e.g. \cite[Theorem 6.5]{loos}):
\begin{itemize}
\item the extremal points of $\overline{D}$ (i.e those that do not lie in the interior of any segment contained in $\overline{D}$);
\item the points on the boundary $\partial D$ attaining the largest distance from the origin with respect to the Bergman metric (\cite[\S VIII.3]{helg} or \cite[\S 1]{loos});
\item the manifold of {\it maximal tripotents} in the Jordan triple system attached to $D$ (notions we will not recall here, referring to \cite[\S\S 3 and 6]{loos} instead).
\end{itemize}

Attached to all of this ``classical'' geometry is the non-commutative {\it Toeplitz $C^*$-algebra} $\cT(D)$, generated by the Hardy-Toeplitz operators with continuous symbol on the {\it Hardy Hilbert space} $H^2(\check{S}(D))$; see e.g. \cite{bc-wh,bck-wh,cob-sing,up0,up-alg,up-bk,up-surv}, where the structure of the $C^*$-algebras $\cT(D)$ is studied (and the pertinent concepts recalled).

Part of \cite{chi-td}, which motivates the present note, is concerned with reconstructing the domain $D$ from its non-commutative counterpart, i.e. the $C^*$-algebra $\cT(D)$.  This is always possible, in the following sense (\cite[Theorem 3.9]{chi-td}):

\begin{theorem}\label{th:diiso}
  The Hardy-Toeplitz $C^*$-algebras $\cT(D_i)$ associated to two bounded symmetric domains are stably isomorphic if and only if $D_i$ are isomorphic. 
\end{theorem}

One piece of information that is immediately extractable from $\cT(D)$ is the Shilov boundary $\check{S}(D)$ as a topological space: by \cite[Theorem 3.12]{up-alg} $\check{S}(D)$ is homeomorphic to the space of characters (i.e. multiplicative continuous functionals) of $\cT(D)$. In light of this, it would be natural to try to approach \Cref{th:diiso} by recovering $D$ from $\check{S}(D)$ alone, divorcing the discussion from all of its non-commutative-topology baggage:

\begin{question}
  Suppose the Shilov boundaries $\check{S}(D_i)$ of two bounded symmetric domains $D_i$, $i=1,2$ are homeomorphic. Is it the case that $D_i$ are isomorphic?
\end{question}

In this generality the answer is negative, as we observe in \Cref{re:notforred} below. Nevertheless, {\it irreducible} bounded symmetric domains (i.e those that do not decompose non-trivially as Cartesian products of such; \Cref{def:bdd}) are better-behaved in this regard: the main result of the note is

\begin{theorem}[\Cref{th:shildet}]
  If two bounded irreducible symmetric domains have homotopy-equivalent Shilov boundaries then they are isomorphic.
\end{theorem}

The proof is in \Cref{se:main}, after gathering some background in \Cref{se:prel}. As a separate matter, in \Cref{se:fundtors} we note that

\begin{theorem}[paraphrase of \Cref{th:notor}]
  Shilov boundaries of bounded symmetric domains have torsion-free fundamental groups. 
\end{theorem}

\subsection*{Acknowledgements}

I am grateful for insightful comments from H. Upmeier and J. Wolf on these and related matters.

This work is partially supported by NSF grant DMS-1801011.

\section{Preliminaries}\label{se:prel}

The terminology is standard (e.g. \cite[\S VIII.7]{helg}):

\begin{definition}\label{def:bdd}
  A {\it bounded domain} is an open, bounded connected subset $D$ of a complex vector space $V\cong \bC^n$.

  A bounded domain $D$ is {\it symmetric} if every point $p\in D$ is an isolated fixed point of an involutive holomorphic automorphism of $D$.

  A bounded symmetric domain is {\it irreducible} if it does not decompose as a Cartesian product of positive-dimensional bounded-symmetric domains.
\end{definition}

Arbitrary bounded symmetric domains always decompose as Cartesian products of irreducible ones (\cite[Chapter VIII, Proposition 5.5]{helg}). In turn, by \cite[Chapter VIII, Theorem 7.1]{helg} the irreducible bounded symmetric domains are precisely the irreducible, non-compact {\it Hermitian symmetric spaces}, i.e. (\cite[Chapter VIII, Theorem 6.1]{helg}) the spaces of the form $G/K$ where
\begin{itemize}
\item $G$ is a non-compact, connected, simple Lie group of adjoint type (i.e. with trivial center);
\item $K\subset G$ is a maximal compact subgroup (which will then automatically have center isomorphic to the circle group $\bS^1$ \cite[Chapter VIII, Proposition 6.2]{helg}).
\end{itemize}

All of our bounded symmetric domains $D\subset V$ are assumed {\it circled} or {\it circular} in the sense that they
\begin{itemize}
\item contain the origin of $V$, and
\item are invariant under scaling by the modulus-1 complex numbers $\bS^1\subset \bC$. 
\end{itemize}
A non-compact Hermitian symmetric space can always be realized as a circular bounded symmetric domain, uniquely up to linear isomorphism; this is (essentially) the celebrated Harish-Chandra embedding theorem, due originally to \'E. Cartan \cite{cart-bdd} through by case-by-case verification: see e.g. \cite{hc6}, \cite[Theorem 1.6]{loos} or \cite[Chapter II, \S 4, especially Remarks 1 and 2 therein]{sat}.

With respect to the realization $D\subset V$ one can speak about the closure $\overline{D}$, the boundary $\partial D:=\overline{D}\setminus D$, and, most importantly for the discussion below, the Shilov boundary $\check{S}(D)$ (\Cref{def:shil}).

\subsection{Shilov boundaries}\label{subse:shil}

We describe the Shilov boundaries of the various irreducible symmetric domains, based on information available in \cite{viv}. The parameter ranges are chosen as they are in order to avoid low-rank isomorphisms, as listed in \cite[Table 7]{viv}.

\subsubsection{Type $I_{p,q}$, $p\ge q\ge 1$.}\label{sss:1}

It follows from \cite[Theorem 2.41]{viv} (see also \cite[Example 1.5.51]{up-bk}) that the manifold of maximal tripotents consists of $p\times q$ matrices $U$ of maximal rank with $UU^*U=U$, i.e. {\it isometries} of eh Hilbert space $\bC^q$ into the Hilbert space $\bC^p$.

Specifying such an isometry means selecting an orthonormal frame of $q$ vectors in $\bC^p$, i.e. the Shilov boundary is the {\it complex Stiefel manifold} denoted by $W_{p,q}$ in \cite[\S 9]{bor}. It is homeomorphic to the homogeneous space $U(p)/U(p-q)$. 

\subsubsection{Type $II_n$, $n\ge 5$.}\label{sss:2}

Here the bounded symmetric domain consists of skew-symmetric $n\times n$ matrices \cite[Theorem 2.41]{viv}, the action of the unitary group $K=U(n)$ is the usual one on bilinear forms, with $U\in U(n)$ acting by $U\cdot U^t$, and the Shilov boundary is an orbit of some maximal-rank skew-symmetric matrix. This means that we can identify $\check{S}(D)$ with
\begin{itemize}
\item $U(2q)/Sp(q)$ for even $n=2q$, where $Sp(q)$ is the compact symplectic group of $2q\times 2q$ matrices;
\item $U(2q+1)/Sp(q)\times \bS^1$ for odd $n=2q+1$. 
\end{itemize}

\subsubsection{Type $III_n$, $n\ge 2$.}\label{sss:3}

This is analogous to type $II$, except the matrices are now symmetric rather than skew-symmetric \cite[Theorem 2.41]{viv}. $\check{S}(D)$ is the orbit through the unitary group $U(n)$ of a maximal-rank symmetric matrix, so is identifiable with $U(n)/O(n)$. This is concurs with the listing of $U(n)/O(n)$ in \cite[p.131 (iii)]{fkklr}.

\subsubsection{Type $IV_n$, $n\ge 5$.}\label{sss:4}

These are the so-called {\it Lie balls}, and their respective Shilov boundaries, as worked out for instance in \cite[Example 1.5.52]{up-bk}, are the {\it Lie Spheres}
\begin{equation}\label{eq:liesp}
  \bL^n:=\bT\cdot\bS^{n-1}\subset \bC^n,
\end{equation}
where
\begin{itemize}
\item $\bT\subset \bC$ denotes the circle of modulus-1 complex numbers;
\item $\bS^{n-1}\subset \bR^n\subset \bC^n$ denotes the unit sphere embedded as usual, and  
\item `$\cdot$' means scaling. 
\end{itemize}
Note that \Cref{eq:liesp} is not quite a direct product, because for any $z\in \bT$ and $x\in \bS^{n-1}$ the points $z\cdot x$ and $(-z)\cdot (-x)$ coincide. In other words
\begin{equation}\label{lnquot}
  \bL^n\cong (\bS^1\times \bS^{n-1})/(\bZ/2),
\end{equation}
where the $\bZ/2$-action is antipodal on both factors. In fact, we have

\begin{lemma}\label{le:ln}
  $\bL^n$ fibers over the circle $\bS^1\cong \bT/\{\pm 1\}$ with fiber $\bS^{n-1}$. That bundle is trivial precisely when $n$ is even.
\end{lemma}
\begin{proof}
  Consider the non-trivial principal $(\bZ/2)$-bundle over $B\cong \bS^1$ induced by the translation action of
  \begin{equation*}
    \bZ/2\cong \{\pm 1\}\subset \bS^1
  \end{equation*}
  (so the total and base space $E$ and $B$ respectively are both homeomorphic to the circle $\bS^1$). Regarding the sphere $\bS^{n-1}$ as a $(\bZ/2)$-space via the antipodal action, $\bL^n$ is by construction the fiber bundle
  \begin{equation*}
    E\times_{\bZ/2}\bS^{n-1}\to B
  \end{equation*}
  associated to this data (e.g. \cite[Chapter 4, \S 5 and especially Example 5.2]{hus}); by \cite[Theorem 8.2]{steen}, it will be trivial if and only if the corresponding principal bundle
  \begin{equation}\label{eq:princho}
    E\times_{\bZ/2}\cat{homeo}(\bS^{n-1})\to B
  \end{equation}
  is. The classification theorem for bundles over spheres (such as $B\cong \bS^1$; \cite[Theorem 18.3]{steen}) identifies \Cref{eq:princho} with the conjugacy class of the antipodal map in the path-component group $\pi_0(\cat{homeo}(\bS^{n-1}))$. It remains to observe that the antipodal map is
  \begin{itemize}
  \item orientation-preserving, and hence in the trivial path component of $\cat{homeo}(\bS^{n-1})$, for even $n$;
  \item orientation-reversing, and hence in the non-trivial path component otherwise.  
  \end{itemize}
\end{proof}

\begin{remark}
  Although \cite[p.131 (iii)]{fkklr} seems to claim $\bS^1\times \bS^{n-1}$ are Shilov boundaries for type-$IV$ domains, it is easy to see that this cannot be true for odd $n$: in that case the diagonal antipodal action of $\bZ/2$ on $\bS^1\times \bS^{n-1}$ reverses orientation, so the quotient \Cref{lnquot} is not orientable.
\end{remark}

\subsubsection{Type $V$.}\label{sss:5}

In general, we can obtain the Shilov boundary up to finite cover as the group listed as $K$ in \cite[Table 1]{viv} modulo its intersection with the group (for the corresponding type) listed as $L(D^0_{\text{type}})$ in \cite[\S 4.5]{viv}. That intersection is always a maximal compact subgroup in the latter. The reason why this only defines $\check{S}$ up to finite cover is that sources sometimes differ in listing the adjoint versus the simply-connected version of $K$, etc.

In any event, For type $V$ the information is sufficient to deduce the Shilov boundary is universally covered by $SO(10)/SO(7)$. We can do better however (see \Cref{pr:v} below).

\subsubsection{Type $VI$.}\label{sss:6} By the same reasoning as in the previous case, here $\check{S}(D)$ will be covered finitely by
\begin{equation*}
  \bS^1\times E_6/(\text{maximal compact subgroup of the real form }E_6(-26)) = \bS^1\times E_6/F_4,
\end{equation*}
i.e. the Cartesian product between a circle and the exceptional compact symmetric space typically denoted by $EIV$ (e.g. \cite[\S X.6, Table V]{helg}). Here $E_6$ denotes the simply-connected compact form of that exceptional Lie algebra, and the notation $E_6(-26)$ is explained in \cite[\S X.6.2]{helg}.

See also \cite[p.131 (iii)]{fkklr}, where the Shilov boundary in this case is displayed as $\bT\cdot E_6/F_4$; the `$\cdot$' symbol means ``almost Cartesian product'', i.e. Cartesian product up to finite covering.

\section{The main result}\label{se:main}

\begin{theorem}\label{th:shildet}
  The homotopy equivalence class of the Shilov boundary $\check{S}(D)$ completely determines the isomorphism class of the irreducible bounded symmetric domain $D$.
\end{theorem}

\begin{remark}\label{re:notforred}
  Note that irreducibility is crucial in \Cref{th:shildet}, even if we strengthen the relation to homeomorphism (rather than homotopy equivalence): by \Cref{le:ln} the irreducible type-$IV_{2n}$ domain has Shilov boundary homeomorphic to $\bS^1\times \bS^{2n-1}$, as does the reducible type $I_{1,1}\times I_{n,1}$.
\end{remark}

The proof of \Cref{th:shildet} will require some preparation. First, note that the dimension of the Shilov boundary is detectable by the homotopy equivalence class, so is an invariant we can use to distinguish between the various domains; this follows from the presumably well-known

\begin{lemma}\label{le:dimdet}
  Homotopy-equivalent compact boundary-less manifolds have the same dimension.
\end{lemma}
\begin{proof}
  If $M$ is a compact manifold without boundary \cite[Chapter VI, Theorem 7.8 and Corollary 7.14]{bred} show that its dimension is the largest degree in which its integral cohomology $H^*(M,\bZ)$ is non-zero.
\end{proof}

Secondly, we can work separately within (non-)tube type: according to \cite[Theorem 4.11]{kw} an irreducible $D$ is
\begin{itemize}
\item of tube type precisely when the rank of (the finitely generated abelian group) $\pi_1(\check{S}(D))$ is 1;
\item of non-tube type if that rank vanishes (i.e. $\pi_1(\check{S}(D))$ is finite abelian).
\end{itemize}

 For that reason, the homotopy type of $\check{S}:=\check{S}(D)$ determines membership to the (non-)tube class, and it suffices to tackle \Cref{th:shildet} separately in those two cases. For tube domains, recall the following numerical data (see also \cite[Appendix]{chi-td}): 

\begin{table}[h]
  \centering
  \begin{tabular}{|c|c|c|}
    \hline
    type & $\dim_{\bR}$ & rank \\ \hline
    $I_{q,q},\ q\ge 1$ & $2q^2$ & $q$ \\ \hline
    $II_{2q},\ q\ge 3$ & $2q(2q-1)$ & $q$ \\ \hline
    $III_q,\ q\ge 2$ & $q(q+1)$ & $q$ \\ \hline
    $IV_q,\ q\ge 5$ & $2q$ & 2\\ \hline
    $VI$ & 54 & 3 \\ \hline
  \end{tabular}
  \caption{Tube-type domains}
  \label{tab:tube}
\end{table}

The dimension of the Shilov boundary $\check{S}$ is half that of $D$ \cite[Theorem 4.9]{kw}, so it can be read off from the dimension column in \Cref{tab:tube}.

\pff{tube}{th:shildet}{: tube type}
\begin{tube}
  The dimension of $\check{S}$ will suffice to distinguish within each numerical type ($I$, $II$, etc.) by \Cref{le:dimdet}, so we have to differentiate {\it between} numerical types. We do this in a number of steps.

  {\bf Type $VI$ differs from the rest.} From dimensional considerations (i.e. $27=\frac{54}2$ is neither a square nor triangular) the only possibility would be for $VI$ and $IV_{27}$ to have homotopy-equivalent Shilov boundaries. Keeping in mind the descriptions of the Shilov boundaries recalled in \Cref{subse:shil}, this is impossible because
  \begin{itemize}
  \item the fundamental groups of $\check{S}(IV_{27})$, which is finitely covered by $\bS^1\times \bS^{26}$, vanish in the range $2..26$, whereas
  \item according to \cite[Introduction]{hir-e6f4} the symmetric space $EIV$ has fundamental group $\pi_9$ of rank $\ge 1$ and hence $\check{S}(VI)$, which admits a finite cover by $\bS^1\times EIV$, has non-trivial $\pi_9$.
  \end{itemize}

  {\bf Type $III$ differs from the rest.} Simply note that from the descriptions of $\check{S}$ in \Cref{subse:shil} it follows that types $I$, $II$ and $IV$ all have vanishing $\pi_2$ whereas $III$ does not, because
  \begin{itemize}
  \item $\pi_2$ vanishes for all Lie groups, in particular for $U(q)$;
  \item $\pi_1(O(q))$ has torsion while $\pi_1(U(q))$ does not, and hence
  \item by the exact portion
    \begin{equation*}
      \cdots\to \pi_2(U(q))\to \pi_2(U(q)/O(q))\to \pi_1(O(q))\to \pi_1(U(q))\to \cdots
    \end{equation*}
    of the long exact homotopy sequence attached (\cite[Theorem 4.41]{hatch}) to the fibration of $U(q)$ by $O(q)\subset U(q)$ the range of the middle map must contain the (non-trivial) torsion of $\pi_1(O(q))$. 
  \end{itemize}

  {\bf The other cases: tube types $I$, $II$ and $IV$.} The integral cohomology rings of the Shilov boundaries listed in \Cref{sss:1,sss:2} for types $I$ and $II$ are cataloged in \cite[\S 9 and \S 31]{bor}:

  \begin{itemize}
  \item $H^*(U(q))$ (over $\bZ$ or, say, $\bR$) is a skew-symmetric algebra on generators of degrees $1,3,\cdots,4q-1$;
  \item $H^*(U(2p)/Sp(p))$ (similarly, over $\bZ$ or $\bR$) is a skew-symmetric algebra on generators of degrees $1,5,\cdots,4p-3$ so since for us $p\ge 3$, this means at least $1,5,9,\cdots$.
  \end{itemize}
  On the other hand, since the Lie ball $\bL^n$ admits a finite cover by $\bS^1\times \bS^{n-1}$, every finite cover of $\bL^n$ with fundamental group $\bZ$ is homeomorphic to that Cartesian product, and hence has Poincar\'e polynomial $(1+t)(1+t^{n-1})$. It is clear now that no two of these cohomology algebras can be isomorphic (as graded algebras) for the ranges listed in \Cref{tab:tube}.
\end{tube}

The non-tube analogue of \Cref{tab:tube} is 

\begin{table}[h]
  \centering
  \begin{tabular}{|c|c|c|c|}
    \hline
    type & $\dim_{\bR}$ & rank & $\dim ~\text{Shilov boundary}$\\ \hline
    $I_{p,q},\ p>q\ge 1$ & $2pq$ & $q$ & $2pq-q^2$\\ \hline
    $II_{2q+1},\ q\ge 2$ & $2q(2q+1)$ & $q$ & $2q^2+3q$\\ \hline
    $V$ & 32 & 2 & 24\\ \hline    
  \end{tabular}
  \caption{Non-tube-type domains}
  \label{tab:ntube}
\end{table}

\pff{ntube}{th:shildet}{: non-tube type}
\begin{ntube}
  The proof follows the same pattern as before, successively eliminating possible coincidences between the distinct numerical types in \Cref{tab:ntube}.

  {\bf Type $V$ differs from the rest.} Most cases are eliminated by dimension considerations: 24 cannot be written as $2q^2+3q$, so the only possibilities would be to have $\check{S}(V)$ homotopy equivalent to $\check{S}(I_{p,q})$ for $(p,q)=(5,4)$ or $(7,2)$. Either way, the universal cover $SO(10)/SO(7)$ of $\check{S}(V)$ would have to be homotopy-equivalent to the already-simply-connected complex Stiefel manifold $U(p)/U(p-q)$; this cannot happen:
  \begin{itemize}
  \item the complex Stiefel manifolds have torsion-free integral cohomology (\cite[Proposition 9.1]{bor}), whereas
  \item the real Stiefel manifold $SO(10)/SO(7)$ has the same integral cohomology as $SO(9)/SO(7)$ by \cite[Proposition 10.4]{bor}, and in turn that abelian group has torsion by \cite[Proposition 10.1]{bor}.
  \end{itemize}

  {\bf Conclusion: distinguishing non-tube types $I$ and $II$.} As noted in \Cref{sss:1,sss:2}, these are (via \Cref{tab:ntube})
  \begin{itemize}
  \item The complex Stiefel manifolds $U(p)/U(p-q)$ for $p>q\ge 1$, and 
  \item The homogeneous spaces $U(2r+1)/Sp(r)\times U(1)$.
  \end{itemize}
  We will again use cohomological invariants to tell them apart. The former, by \cite[Proposition 9.1]{bor}, has for its integral cohomology ring the exterior algebra generated by elements of odd degrees spanning $2(p-q)+1$ up to $2p-1$:
  \begin{equation}\label{eq:fori}
    H^*(\check{S}(I_{p,q}))\cong H^*(\bS^{2(p-q)+1}\times\cdots\times \bS^{2p-1}). 
  \end{equation}
  As for the former, we will have to supplement the methods of \cite{bor} slightly. Set
  \begin{equation*}
    B=U(2r+1)/Sp(r)\times U(1),
  \end{equation*}
  our target space. It is the base of a fibration with circle fiber $U(1)\cong \bS^1$ and total space
  \begin{equation*}
    E:=U(2r+1)/Sp(r).
  \end{equation*}
  The latter is treatable by the technique applied in \cite[Proposition 31.3]{bor} to the computation of the cohomology $H^*(U(2n)/Sp(n))$: the fiber $Sp(r)$ is totally non-homologous to zero in the sense of \cite[discussion preceding Proposition 4.1]{bor} (for integral cohomology), and, as in the proof of \cite[Proposition 31.3]{bor}, it follows that
  \begin{equation}\label{eq:tot}
    H^*(E)\cong H^*(\bS^{1}\times \bS^{5}\times\cdots\times\bS^{4r-3}\times \bS^{4r+1}). 
  \end{equation}

  On the other hand, the long exact homotopy sequence associated to the fibration of $U(2r+1)$ over $B$ shows that both $\pi_1$ and $\pi_2$ vanish for the latter, and hence so does its $2^{nd}$ integral (co)homology (e.g. by the Hurewicz theorem and the universal coefficient theorem: \cite[Theorems 4.32 and 3.2]{hatch}). But then the Chern class of the circle bundle $E\to B$ is trivial, and hence the bundle itself must be trivial; in other words,
  \begin{equation*}
    E\cong B\times \bS^1. 
  \end{equation*}
  It follows from the K\"unneth theorem (\cite[Theorem 3.16]{hatch}) and \Cref{eq:tot} that 
  \begin{equation*}
    H^*(B)\cong H^*(\bS^{5}\times\cdots\times\bS^{4r-3}\times \bS^{4r+1}). 
  \end{equation*}
  Since $r\ge 2$ this means at least two spheres, whose dimensions form an arithmetic progression with step 4, whereas the dimensions of the spheres in \Cref{eq:fori} form a progression with step 2. This shows that the two cohomology rings cannot be isomorphic, finishing the proof.
\end{ntube}

\section{An aside on fundamental-group torsion}\label{se:fundtors}

Although this is not central to the main discussion above, the present section records the fact that Shilov boundaries have torsion-free fundamental groups. To my knowledge this is not noted in the literature (some of which explicitly leaves open the possibility of having torsion; e.g. \cite[Theorem 4.11]{kw}).

The main result of the section is

\begin{theorem}\label{th:notor}
  Let $D$ Ne an irreducible bounded symmetric domain and $\check{S}:=\check{S}(D)$ its Shilov boundary. The fundamental group $\pi_1(\check{S})$ is
  \begin{enumerate}[(a)]
  \item\label{item:1} $\bZ$ when $D$ is of tube type;
  \item\label{item:2} trivial otherwise.
  \end{enumerate}
\end{theorem}
\begin{proof}

  {\bf \Cref{item:1}: tube type.} The conclusion follows obliquely from a number of remarks in the literature. According to \cite[Proposition 1.5.75 and surrounding discussion]{up-bk} there is a fibration
  \begin{equation}\label{eq:nfib}
    N:\check{S}\to \bS^1.
  \end{equation}
  $N$ is the {\it norm function} or {\it generalized determinant} of the Jordan algebra associated to $D$ (see e.g. \cite[Theorem 3.7, Lemma 3.8]{up-jord}, \cite[Proposition 1.6 and discussion preceding it]{up0}, or again \cite[\S 1.5]{up-bk}; in the latter $N$ is denoted by $\Delta$).

  The fiber of the bundle \Cref{eq:nfib} is the {\it reduced Shilov boundary}
  \begin{equation*}
    \check{S}':=\{x\in \check{S}\ |\ N(x)=1\}.
  \end{equation*}
  In turn, the latter is connected and simply-connected by \cite[p.63]{up-bk} (see also \cite[Proposition 2.2]{badi}, which cites the same source). The claim that $\pi_1(\check{S})\cong \bZ$ now follows from the homotopy exact sequence
  \begin{equation*}
    \cdots\to \pi_1(\check{S}')\to \pi_1(\check{S})\to \pi_1(\bS^1)\to \pi_0(\check{S}')\to \cdots,
  \end{equation*}
  which shows that the norm $N:\check{S}\to \bS^1$ induces an isomorphism of fundamental groups.

  {\bf \Cref{item:2}: non-tube type.} The non-tube domains are listed in \Cref{tab:ntube} and their boundary structures recalled in \Cref{sss:1,sss:2,sss:4}. For types $I$ and $II$ $\check{S}$ can be realized as a homogeneous space $K/L$ for compact connected Lie groups $L\subseteq K$ whose inclusion induces an isomorphism
  \begin{equation*}
    \pi_1(L)\cong \pi_1(K)
  \end{equation*}
  (e.g. for type $I_{p,q}$, $p>q$ we have $K=U(p)$ and $L=U(p-q)$ embedded as upper-right-corner matrices). The conclusion then follows from the portion
  \begin{equation*}
    \cdots\to \pi_1(L)\to \pi_1(K)\to \pi_1(K/L)\to \pi_0(L)\to \cdots
  \end{equation*}
  of the long exact homotopy sequence associated to the relevant fibration.

  We treat type $V$ separately in \Cref{pr:v}.
\end{proof}

\begin{proposition}\label{pr:v}
  The Shilov boundary $\check{S}$ of the irreducible bounded symmetric domain $D:=D_V$ of type $V$ is homeomorphic to $SO(10)/SO(7)$.
\end{proposition}
\begin{proof}
  We already know that $\check{S}$ is universally covered by the Stiefel variety $V_{10,3}:=SO(10)/SO(7)$ (notation as in \cite[\S 10]{bor}). It thus remains to prove that $\check{S}$ is simply-connected.

  Denote
  \begin{equation*}
    \pi\cong \pi_1(\check{S}),
  \end{equation*}
  a finite abelian group according to \cite[Theorem 4.11]{kw}. Since the boundary $\partial D$ is homeomorphically a sphere $\bS^{31}$, the Alexander duality theorem (\cite[Chapter VI, Corollary 8.7]{bred}) for (co)homology with $\pi$ coefficients ensures that
  \begin{equation}\label{eq:aldual}
    \End(\pi)\cong H^1(\check{S},\pi)\cong H_{29}(\partial D\setminus \check{S},\pi). 
  \end{equation}
  In order to prove that $\pi$ is trivial it is thus enough to argue that \Cref{eq:aldual} vanishes. To see this, note that by \cite[Remark 6.9]{loos} and the description of $D_V^1$ in \cite[\S 4.5.5]{viv} the complement $Y:=\partial D\setminus \check{S}$ is the total space of a fibration over some compact manifold $X$, with the bounded symmetric domain $D_{I_{5,1}}$ as a fiber.

  Since the 10-dimensional manifold $D_{I_{5,1}}$ is contractible, $Y$ is homotopy-equivalent to the 21-dimensional manifold $X$. In particular, its degree-29 homology \Cref{eq:aldual} vanishes; as outlined above, this finishes the proof.
\end{proof}

\begin{remark}
  In Jordan-theoretic language, the manifold $X$ consists of the rank-1 tripotents in the Jordan triple system associated to $D_V$ (\cite[\S 5]{loos}). This is what in \cite[Chapter VIII, (8.19)]{sat} would be denoted by $K^{\circ}o^{(1)}$ (the `$(1)$' superscript denotes the rank). Loc.cit. provides the dimension of that manifold (see also \cite[Chapter VIII, p.136]{sat} for a definition of the symbol $\widetilde{V}_b$): it is the sum
  \begin{equation*}
    \dim U + \dim V = 21
  \end{equation*}
  for $U$ and $V$ listed under type $V$, $b=1$ on \cite[p.117]{sat}. This provides a numerical sanity check for the computed $\dim X=21$ in the proof of \Cref{pr:v}.
\end{remark}


\bibliographystyle{plain}

\Addresses

\end{document}